\documentclass[11pt, twoside]{article}
\usepackage{amssymb, amsmath, amsthm}
\usepackage[left=25mm, right=25mm, bottom=30mm]{geometry}
\usepackage{fancyhdr}
\usepackage{tikz}
\usetikzlibrary{positioning}
\usepackage{titling}
\predate{}
\postdate{}
\usepackage{lipsum}
\newtheorem{theorem}{Theorem}
\newtheorem{lemma}[theorem]{Lemma}
\newtheorem{proposition}[theorem]{Proposition}
\newtheorem {conjecture}[theorem]{Conjecture}
\begin{document}
\bibliographystyle{amsplain}
\pagestyle{fancy}
\fancyhf{}
\fancyhead [LE, RO] {\thepage}
\fancyhead [CE] {MARIA-ROMINA IVAN}
\fancyhead [CO] {SATURATION FOR THE BUTTERFLY POSET}
\renewcommand{\headrulewidth}{0pt}
\title{\Large\textbf{SATURATION FOR THE BUTTERFLY POSET}}
\author{MARIA-ROMINA IVAN}
\date{}
\maketitle
\begin{abstract}Given a finite poset $\mathcal P$, we call a family $\mathcal F$ of subsets of $[n]$ $\mathcal P$-saturated if $\mathcal F$ does not contain an induced copy of $\mathcal P$, but adding any other set to $\mathcal F$ creates an induced copy of $\mathcal P$. The induced saturated number of $\mathcal P$, denoted by $\text{sat}^*(n,\mathcal P)$, is the size of the smallest $\mathcal P$-saturated family with ground set $[n]$.\\In this paper we are mainly interested in the four-point poset called the butterfly. Ferrara, Kay, Kramer, Martin, Reiniger, Smith and Sullivan showed that the saturation number for the butterfly lies between $\log_2{n}$ and $n^2$. We give a linear lower bound of $n+1$. We also prove some other results about the butterfly and the poset $\mathcal N$.
\end{abstract}
\section{Introduction}

\par We say that a poset $(\mathcal P,\preceq)$ contains an \textit{induced copy} of a poset $(\mathcal Q,\preceq')$ if there exists an injective order-preserving function $f:\mathcal Q\rightarrow\mathcal P$ such that $(f(\mathcal Q),\preceq)$ is isomorphic to $(\mathcal Q,\preceq')$. If elements $a$ and $b$ of a poset are not related, or incomparable, we write $a\parallel b$. \par In this paper we consider the power-set of $[n]=\lbrace 1,2,\cdots,n\rbrace$ with the partial order induced by inclusion. 
If $\mathcal Q$ is a finite poset and $\mathcal F$ is a family of subsets of $[n]$, we say that $\mathcal F$ is $\mathcal Q$ \textit{saturated} if $\mathcal F$ does not contain an induced copy of $\mathcal Q$, and for any $S\notin\mathcal F$, $\mathcal F\cup S$ contains an induced copy of $\mathcal Q$. The smallest size of a $\mathcal Q$-saturated family of subsets of $[n]$ is called the \textit{induced saturated number}, denoted by $\text{sat}^*(n,\mathcal Q)$.
\par The posets for which we will analyse the induced saturated number are the butterfly (Figure 1) which we denote by $\mathcal B$, and $\mathcal N$ (Figure 2).
\begin{center}
\begin{tikzpicture}
\node (A) at (-4,0) {\textbullet};
\node (B) at (-4,2) {\textbullet};
\node (C) at (-2,0) {\textbullet};
\node (D) at (-2,2) {\textbullet};
\draw (A)--(B)--(C)--(D)--(A); 
\node (A_1) at (0,0) {\textbullet};
\node (B_1) at (0,2) {\textbullet};
\node (C_1) at (2,0) {\textbullet};
\node (D_1) at (2,2) {\textbullet};
\draw (A_1)--(B_1)--(C_1)--(D_1);
\node (f1) at (-3,-1) {$\text{Figure 1: }\mathcal B$};
\node (f2) at (1,-1) {$\text{Figure 2: }\mathcal N$};
\end{tikzpicture}
\end{center}
Ferarra, Kay, Krammer, Martin, Reiniger, Smith and Sullivan \cite{Ferrara2017TheSN} proved that for a large class of posets, including the butterfly and $\mathcal N$, the induced saturated number is at least the biclique  cover number of the complete graph on $n$ vertices, namely $\log_2 n$. Based on the butterfly-saturated family $\mathcal F=\lbrace \emptyset, \lbrace i\rbrace, \lbrace i,j\rbrace, \lbrace 1,2,\cdots,i\rbrace: 1\leq i\leq n, 1\leq j\leq n\rbrace$, they also conjectured that the induced saturated number for the butterfly is of order $n^2$. They also provided an upper bound of $2n$ for $\text{sat}^*(n,\mathcal N)$ by constructing the $\mathcal N$-saturated family $\lbrace \emptyset, \lbrace i\rbrace, \lbrace 1,2,\cdots,i\rbrace:1\leq i\leq n\rbrace$. In this paper we improve the lower bounds to $n+1$ for the butterfly and to $\sqrt n$ for $\mathcal N$.\\
We mention that there has been a considerable amount of work on extremal problems in poset theory--see Griggs and Li \cite{inbook} for a survey. Saturation for posets was introduced by Gerbner and Patk\'os \cite{gerbner2018extremal}, although this was not for $\textit{induced}$ saturation. See Morrison, Noel and Scott \cite{Morrison2014OnSK} for some improvements and related results.
\section{A linear lower bound on $\text{sat}^*(n, \mathcal B)$}
In this section we prove that any butterfly-saturated family has size at least $n+1$. The strategy is to look at singletons that are not in the family and associate in an injective manner to each one of them a butterfly, satisfying some maximality conditions, that is formed when the singleton is added to the family. This injective association will allow us to construct an explicit injection from the ground set $[n]$ to $\mathcal F$. Together with the observation that the empty set has to be in $\mathcal F$, we will obtain the claimed lower bound.
\begin{lemma}Let $\mathcal{F}$ be a $\mathcal B$-saturated family. If $\lbrace i\rbrace$ and $\lbrace j\rbrace \in \mathcal{F}$, then the pair $\lbrace i,j\rbrace$ is an element of $\mathcal{F}$.
\begin{proof} Assume $\lbrace i,j\rbrace$ is not an element of $\mathcal F$. Since $\mathcal F$ is $\mathcal B$-saturated, this implies that $\mathcal{F}\cup \lbrace i,j\rbrace$ contains a butterfly. That butterfly must involve the pair $\lbrace i,j\rbrace$, or else the initial family will contain a butterfly.\\If $\lbrace i, j\rbrace$ is one of the maximal elements of the butterfly, then the two incomparable elements below it must be the singletons $\lbrace i\rbrace$ and $\lbrace j\rbrace$. But then they will be included in the other maximal element of the butterfly, call it $M$, and thus $\lbrace i,j\rbrace\subset M$, contradicting the incomparability of the maximal elements.\\If $\lbrace i, j\rbrace$ is one of the minimal elements, then, by replacing $\lbrace i,j\rbrace$ in the newly formed butterfly with $\lbrace i\rbrace$ or $\lbrace j\rbrace$, we form a butterfly in $\mathcal{F}$, unless $\lbrace i\rbrace$ and $\lbrace j\rbrace$ are comparable to the other minimal element, call it $N$. Thus $\lbrace i,j\rbrace\subset N$, contradicting $N\parallel \lbrace i,j\rbrace$.
\end{proof}
\end{lemma}
\par Note that if a butterfly-saturated family $\mathcal F$ contains $\Theta(n)$ singletons, then $\mathcal F$ contains $\Theta(n^2)$ pairs, so that $|\mathcal F|\geq\Theta(n^2)$.\\\par
We define a \textit{chevron} to be a triplet $(A,B,C)$ of subsets of $\lbrack n\rbrack$ with the property that $C\subset A$, $C\subset B$ and $A\parallel B$.
\begin{theorem}
$\text{sat}^*(n,\mathcal B)\geq n+1$.
\begin{proof}
We will first assign to every $\lbrace i\rbrace\notin F$ a chevron with elements from $\mathcal F$ in such a way that no two singletons are assigned the same chevron.\\
If $\lbrace i\rbrace\notin F$, then $\mathcal F\cup\lbrace i\rbrace$ contains a butterfly and that butterfly has to involve the singleton, otherwise $\mathcal F$ would not be butterfly-free. Moreover, that singleton has to be one of the minimal elements of the butterfly since it does not have two incomparable elements below it. Therefore we have the structure shown below.

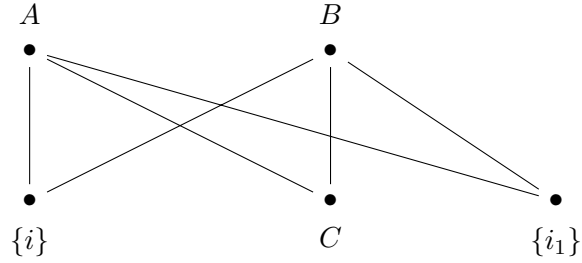
\begin{center}
\begin{tikzpicture}
\node (A) at (-2,0) {\textbullet};
\node [above=0cm of A] {$A$};
\node (B) at (2,0) {\textbullet};
\node [above=0cm of B] {$B$};
\node (extra) at (-2, -2) {\textbullet};
\node [below=0cm of extra] {$\lbrace i\rbrace$};
\node (C) at (2, -2) {\textbullet};
\node [below=0cm of C] {$C$};
\node (extrab) at (5, -2) {\textbullet};
\node [below=0cm of extrab] {$\lbrace i_1\rbrace$};
\draw (A)--(C)--(B)--(extra)--(A)--(extrab)--(B);
\end{tikzpicture}
\end{center}
It is obvious that $i\notin C$. Among all these constructions, we pick the one having $|C|$ maximal and assign the chevron $(A,B,C)$ to $\lbrace i\rbrace$. We now need to show that under this construction, a chevron is not assigned to two different singletons. Assume that $\lbrace i_1\rbrace$ has also been assigned the same $(A,B,C)$ chevron, as shown above. Consider the set $C\cup\lbrace i_1\rbrace$. It is clearly incomparable to $\lbrace i\rbrace$ since $i\neq i_1$, it is contained in both $A$ and $B$, but not equal to either of them since they contain $i$, thus $C\cup\lbrace i_1\rbrace\notin F$ by maximality of $|C|$. Therefore $C\cup\lbrace i_1\rbrace$ has to form a butterfly with three elements of $\mathcal F$.
\begin{enumerate}
\item Case 1: $C\cup\lbrace i_1\rbrace$ is one of the minimal elements of the butterfly as shown below	, where $A^*, B^*, C^*\in\mathcal F$.
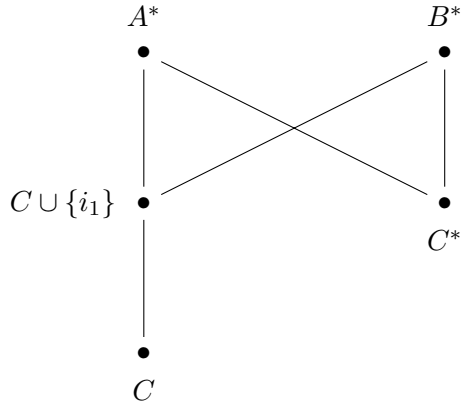
\begin{center}
\begin{tikzpicture}
\node (A) at (-2,0) {\textbullet};
\node [above=0cm of A] {$A^*$};
\node (B) at (2,0) {\textbullet};
\node [above=0cm of B] {$B^*$};
\node (extra) at (-2, -2) {\textbullet};
\node [left=0cm of extra] {$C\cup\lbrace i_1\rbrace$};
\node (C) at (2, -2) {\textbullet};
\node [below=0cm of C] {$C^*$};
\node (extra2) at (-2, -4) {\textbullet};
\node [below=0cm of extra2] {$C$};
\draw (A)--(C)--(B)--(extra)--(A);
\draw (extra)--(extra2);
\end{tikzpicture}
\end{center}
To stop $A^*, B^*, C^*, C$ from forming a butterfly in $\mathcal F$, we need $C$ and $C^*$ to be comparable, and since $C\cup\lbrace i_1\rbrace\parallel C^*$, the only option is $C\subset C^*$ and $i_1\notin C^*$. Now the chevron $(A^*, B^*, C^*)$ has the property that $\lbrace i_1\rbrace\parallel C^*$, $i_1\in A^*, B^*$ and the size of $C^*$ is strictly greater than the size of $C$. By construction, this contradicts that $(A,B,C)$ was assigned $\lbrace i_1\rbrace$.
\item Case 2: $C\cup\lbrace i_1\rbrace$ is one of the maximal elements of the butterfly as shown in the diagram below. 
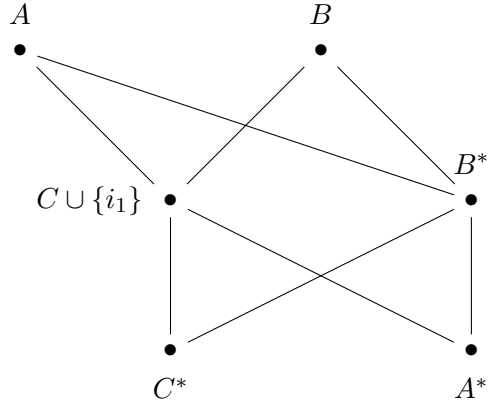
\begin{center}
\begin{tikzpicture}
\node (A) at (2,-2) {\textbullet};
\node [below=0cm of A] {$A^*$};
\node (B) at (2,0) {\textbullet};
\node [above=0cm of B] {$B^*$};
\node (extra) at (-2, 0) {\textbullet};
\node [left=0cm of extra] {$C\cup\lbrace i_1\rbrace$};
\node (C) at (-2, -2) {\textbullet};
\node [below=0cm of C] {$C^*$};
\node (extra2) at (-4, 2) {\textbullet};
\node [above=0cm of extra2] {$A$};
\node (extra3) at (0, 2) {\textbullet};
\node [above=0cm of extra3] {$B$};
\draw (B)--(extra2)--(extra)--(extra3)--(B);
\draw (extra)--(A)--(B)--(C)--(extra);
\end{tikzpicture}
\end{center}
We obviously have $C\cup\lbrace i_1\rbrace\subset A, B$. To stop $A$ (or $B$), $C^*$, $A^*$ and $B^*$ from forming a butterfly in $\mathcal F$, we need both $A$ and $B$ to be comparable to $B^*$ and the only option is $B^*\subset A, B$. We notice that we are now in the previous case where $C\cup\lbrace i_1\rbrace$ is the minimal element of a butterfly, namely the one formed with $A$, $B$ and $B^*$.
\end{enumerate}
We therefore conclude that we can associate every singleton (not in our family) with a chevron, and no two singletons are associated with the same chevron.\\The next step is to show that $C\cup\lbrace i\rbrace\in\mathcal F$ where $C$ is the maximal element of the chevron assigned to the singleton $\lbrace i\rbrace\notin\mathcal F$. Assume that $C\cup\lbrace i\rbrace\notin\mathcal F$ and as before, it will have to form a butterfly with three elements of $\mathcal F$.
\begin{enumerate}
\item Case 1: $C\cup\lbrace i\rbrace$ is one of the minimal elements of the butterfly. Assume $C^*$ is the other minimal element and $A^*, B^*$ are the two maximal incomparable elements. The same argument we used above for $C\cup\lbrace i_1\rbrace$ will tell us that $C\subset C^*$, $i\notin C^*$ and $i\in A^*\cap B^*$, contradicting the maximality of $|C|$.
\item Case 2: $C\cup\lbrace i\rbrace$ is one of the maximal elements of the butterfly, $B^*$ is the other maximal element and $A^*$ and $C^*$ are the two incomparable minimal elements. Since $C\cup\lbrace i\rbrace\subset A,B$, but is not equal to either of them (if for example $A=C\cup\lbrace i\rbrace$, then $A\subset B$ which cannot happen), the same arguments as above will tell us that $B^*\subset A,B$, which leads us back to the first case.  	
\end{enumerate}
Therefore we indeed have $C\cup\lbrace i\rbrace\in\mathcal F$.\\
Let us now define the following function from $[n]$ to elements of $\mathcal F$: \begin{center}
$i\longmapsto\begin{cases}
\lbrace i\rbrace & \text{ if }\lbrace i\rbrace\in\mathcal F\\
C\cup\lbrace i\rbrace & \text{ if }\lbrace i\rbrace\notin\mathcal F,	
\end{cases}
$	
\end{center}
where $C$ is the minimal element of the chevron assigned to the singleton not in the family. By what we just proved above, this is a well-defined function from $[n]$ to $\mathcal F$.\\We claim that this function is an injection. Since $C\cup\lbrace i\rbrace$ cannot be a singleton and the function is obviously injective on singletons, we only need to show that $\bar C\cup\lbrace i_1\rbrace\neq C\cup\lbrace i\rbrace$ if $i\neq i_1$.\\
If $\bar C\cup\lbrace i_1\rbrace=C\cup\lbrace i\rbrace$, then $C\neq \bar C$ since $i\neq i_1$, but they have the same cardinality which tells us that they are incomparable. Let $(A,B,C)$ be the chevron $C$ is originating from. By construction we have that $C\cup\lbrace i\rbrace\subset A,B$. It would then follow that $\bar C\subset A,B$ which immediately implies that $A,B,C,\bar C$ would form a butterfly in $\mathcal F$, contradiction.\\
Hence, the above function is indeed an injection from $[n]$ to non-empty elements of $\mathcal F$. It is easy to see that if we add the empty set to $\mathcal F$, it cannot form a butterly as it is comparable to everything, thus by saturation $\emptyset\in\mathcal F$.\\
We therefor conclude that $|\mathcal F|\geq n+1$ for every butterfly-saturated family, implying $\text{sat}^*(n, \mathcal B)\geq n+1$, as claimed.
\end{proof}
\end{theorem}
\section{Further analysis of singletons}
In Section 2 we looked at singletons that are not in our family and constructed an injection from them to the set of chevrons with elements in $\mathcal F$. Because of the crucial role singletons played in the above proof, it is of interest to see what more can be said about the number of singletons in the family. In this section we use the same techniques and look at pairs that are not in our family. This provides us with better bounds in the case when we have few singletons in the family. 
\begin{theorem} Let $\mathcal F$ be a $\mathcal B$-saturated family containing $k\geq 1$ singletons. Then $|\mathcal F|\geq\binom{k}{2}+k(n-k)$.
\begin{proof} From Lemma 1 we already know that $\mathcal F$ contains the $\binom{k}{2}$ pairs made out of singletons of $\mathcal F$ only. We look at pairs $\lbrace i,j\rbrace\notin\mathcal F$. Also from Lemma 1 we get that at least one of the singletons in this pair is not in our family. We now restrict our attention to pairs $\lbrace i,j\rbrace$ that are not in $\mathcal F$ and for which exactly one of the singletons $\lbrace i\rbrace$ and $\lbrace j\rbrace$ is in $\mathcal F$. As in the previous chapter, we will show that we can uniquely assign a chevron to these pairs and construct an injection from the set of pairs containing at least one singleton of $\mathcal F$, to $\mathcal F$.\\
\\As before, if $\lbrace i,j\rbrace\notin\mathcal F$, then $\mathcal F\cup\lbrace i,j\rbrace$ contains a butterfly involving the pair $\lbrace i,j\rbrace$, and that this pair has to be one of the minimal elements. This is because if it is one of the maximal elements, the only candidates for the two incomparable minimal elements are the singletons made out of its elements. But we know that one must not be in our family and so this cannot happen. We then have a butterfly formed as shown below.\\
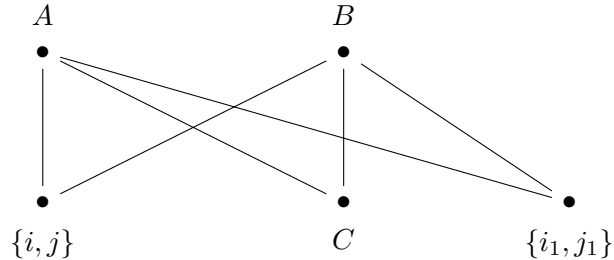
\begin{center}
\begin{tikzpicture}
\node (A) at (-2,0) {\textbullet};
\node [above=0cm of A] {$A$};
\node (B) at (2,0) {\textbullet};
\node [above=0cm of B] {$B$};
\node (extra) at (-2, -2) {\textbullet};
\node [below=0cm of extra] {$\lbrace i, j\rbrace$};
\node (C) at (2, -2) {\textbullet};
\node [below=0cm of C] {$C$};
\node (extrab) at (5, -2) {\textbullet};
\node [below=0cm of extrab] {$\lbrace i_1, j_1\rbrace$};
\draw (A)--(C)--(B)--(extra)--(A)--(extrab)--(B);
\end{tikzpicture}
\end{center}
Among all these configurations, we choose the one having $|C|$ maximal and assign the chevron $(A,B,C)$ to the pair $\lbrace i, j\rbrace$. We have to show that this assignment is an injection for the pairs we are considering.\\
Assume that under this construction, the same chevron has been assigned to a different pair $\lbrace i_1, j_1\rbrace$.
We have the following two cases to analyse: 
\begin{enumerate}
\item $\lbrace i, j\rbrace\parallel C\cup\lbrace i_1, j_1\rbrace$, then $C\cup\lbrace i_1,j_1\rbrace\notin \mathcal F$ by the maximality of the chevron assigned to $\lbrace i, j\rbrace$. This means that $C\cup\lbrace i_1,j_1\rbrace$ will form a butterfly in $\mathcal F$.\begin{enumerate}
\item Case 1: $C\cup\lbrace i_1, j_1\rbrace$ is one of the minimal elements of the butterfly as shown in the diagram below, where $A^*, B^*, C^*\in\mathcal F$. 
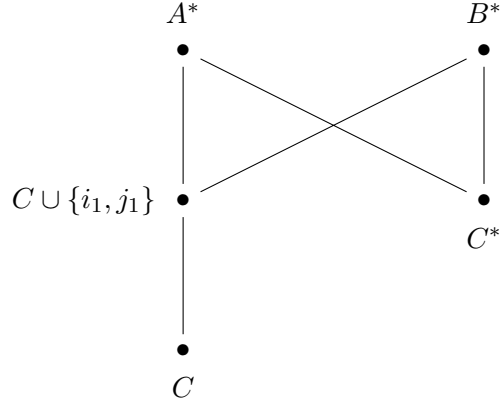
\begin{center}
\begin{tikzpicture}
\node (A) at (-2,0) {\textbullet};
\node [above=0cm of A] {$A^*$};
\node (B) at (2,0) {\textbullet};
\node [above=0cm of B] {$B^*$};
\node (extra) at (-2, -2) {\textbullet};
\node [left=0cm of extra] {$C\cup\lbrace i_1, j_1\rbrace$};
\node (C) at (2, -2) {\textbullet};
\node [below=0cm of C] {$C^*$};
\node (extra2) at (-2, -4) {\textbullet};
\node [below=0cm of extra2] {$C$};
\draw (A)--(C)--(B)--(extra)--(A);
\draw (extra)--(extra2);
\end{tikzpicture}
\end{center}
But $A^*, B^*, C^*, C$ have to not form a butterfly, therefore $C$ and $C^*$ are comparable and thus $C\subseteq C^*$. Because $C^*\parallel C\cup\lbrace i_1, j_1\rbrace$, $C\neq C^*$ and $\lbrace i_1, j_1\rbrace\parallel C^*$. Therefore the chevron $(A^*, B^*, C^*)$ can be assigned to $\lbrace i_1, j_1\rbrace$ and the size of $C^*$ is strictly greater than the size of $C$, which is a contradiction.
\item Case 2: $C\cup\lbrace i_1, j_1\rbrace$ is one of the maximal elements as shown below. 
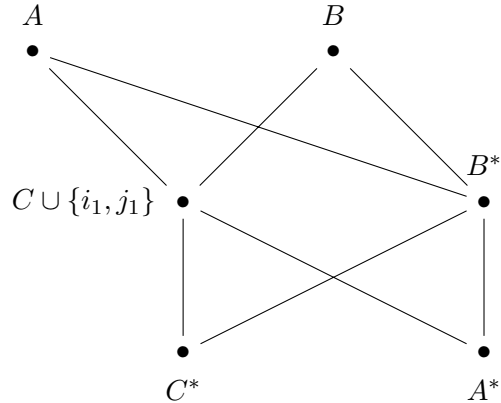
\begin{center}
\begin{tikzpicture}
\node (A) at (2,-2) {\textbullet};
\node [below=0cm of A] {$A^*$};
\node (B) at (2,0) {\textbullet};
\node [above=0cm of B] {$B^*$};
\node (extra) at (-2, 0) {\textbullet};
\node [left=0cm of extra] {$C\cup\lbrace i_1, j_1\rbrace$};
\node (C) at (-2, -2) {\textbullet};
\node [below=0cm of C] {$C^*$};
\node (extra2) at (-4, 2) {\textbullet};
\node [above=0cm of extra2] {$A$};
\node (extra3) at (0, 2) {\textbullet};
\node [above=0cm of extra3] {$B$};
\draw (B)--(extra2)--(extra)--(extra3)--(B);
\draw (extra)--(A)--(B)--(C)--(extra);
\end{tikzpicture}
\end{center}
We obviously have $C\cup\lbrace i_1, j_1\rbrace\subset A, B$. To stop $A$ (or $B$), $C^*$, $A^*$ and $B^*$ from forming a butterfly in $\mathcal F$, we need both $A$ and $B$ to be comparable to $B^*$ and the only option is $B^*\subset A, B$. We notice that we are now in the previous case where $C\cup\lbrace i_1, j_1\rbrace$ is the minimal element of a butterfly, namely the one formed with $A$, $B$ and $B^*$. 
\end{enumerate}
\item $\lbrace i,j\rbrace$ is comparable to $C\cup\lbrace i_1, j_1\rbrace$ and $\lbrace i_1, j_1\rbrace$ is comparable to $C\cup\lbrace i,j\rbrace$. By cardinality, the only options are $\lbrace i, j\rbrace\subset C\cup\lbrace i_1, j_1\rbrace$ and $\lbrace i_1, j_1\rbrace\subset C\cup\lbrace i, j\rbrace$. Because $\lbrace i, j\rbrace\parallel C$, we need $\lbrace i,j\rbrace\cap\lbrace i_1, j_1\rbrace\neq\emptyset$ and wlog $i=i_1$ and $j\neq j_1$. It then follows that $j,j_1\in C$ and thus $i\notin C$. Because $i\notin C$, we cannot have $\lbrace i\rbrace\in\mathcal F$, otherwise $A, B, C, \lbrace i\rbrace$ will form a butterfly in $\mathcal F$. Thus, since the pairs we are analysing consist of exactly one singleton in $\mathcal F$, we obtain that $\lbrace j\rbrace$ and $\lbrace j_1\rbrace$ are elements of $\mathcal F$. But then we obtain $A,B,\lbrace j\rbrace, \lbrace j_1\rbrace$ a butterfly in $\mathcal F$, which is a contradiction. 
\end{enumerate}
We will construct an injection from the set of pairs $\lbrace i,j\rbrace$ with at least one singleton in $\mathcal F$ to elements of $\mathcal F$. We know that we can assign a unique chevron to every such pair that is not in our butterfly-saturated family.\\Let $\lbrace i, j\rbrace$ be such a pair and $(A, B, C)$ its chevron. If $C\cup\lbrace i, j\rbrace\notin \mathcal F$, then it has to form a butterfly when added to the family.
\begin{enumerate}
\item Case 1: $C\cup\lbrace i,j\rbrace$ is one of the minimal elements of the butterfly. 
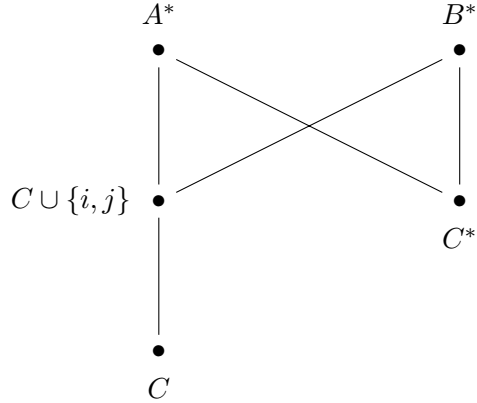
\begin{center}
\begin{tikzpicture}
\node (A) at (-2,0) {\textbullet};
\node [above=0cm of A] {$A^*$};
\node (B) at (2,0) {\textbullet};
\node [above=0cm of B] {$B^*$};
\node (extra) at (-2, -2) {\textbullet};
\node [left=0cm of extra] {$C\cup\lbrace i,j\rbrace$};
\node (C) at (2, -2) {\textbullet};
\node [below=0cm of C] {$C^*$};
\node (extra2) at (-2, -4) {\textbullet};
\node [below=0cm of extra2] {$C$};
\draw (A)--(C)--(B)--(extra)--(A);
\draw (extra)--(extra2);
\end{tikzpicture}
\end{center}
As before, it then follows that $C\subset C^*$ and $\lbrace i,j\rbrace\parallel C^*$, and thus the chevron $(A^*, B^*, C^*)$ could have been assigned to the pair, contradicting the maximality of the minimal element of the chevron.
\item Case 2: $C\cup\lbrace i,j\rbrace$ is one of the maximal elements of the butterfly. It then follows by the same arguments we used before, that $C\cup\lbrace i,j\rbrace$ will be one of the minimal elements in a butterfly containing $A$, $B$ and $B^*$ (as we can see on the diagram below), thus returning to the previous case.
\end{enumerate} 
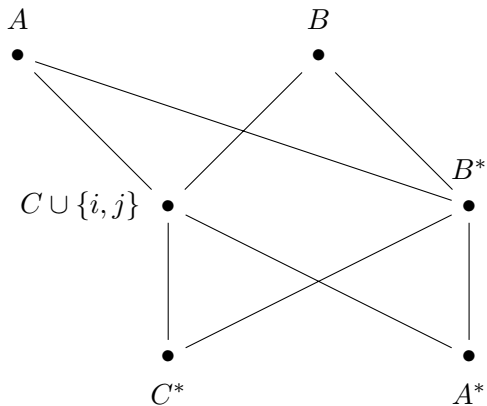
\begin{center}
\begin{tikzpicture}
\node (A) at (2,-2) {\textbullet};
\node [below=0cm of A] {$A^*$};
\node (B) at (2,0) {\textbullet};
\node [above=0cm of B] {$B^*$};
\node (extra) at (-2, 0) {\textbullet};
\node [left=0cm of extra] {$C\cup\lbrace i,j\rbrace$};
\node (C) at (-2, -2) {\textbullet};
\node [below=0cm of C] {$C^*$};
\node (extra2) at (-4, 2) {\textbullet};
\node [above=0cm of extra2] {$A$};
\node (extra3) at (0, 2) {\textbullet};
\node [above=0cm of extra3] {$B$};
\draw (B)--(extra2)--(extra)--(extra3)--(B);
\draw (extra)--(A)--(B)--(C)--(extra);
\end{tikzpicture}
\end{center}
Therefore we must have $C\cup\lbrace i, j\rbrace\in\mathcal F$, and we define the following function from the set of pairs containing at least one singleton in the family:\begin{center}
$\lbrace i, j\rbrace\longmapsto\begin{cases}
\lbrace i, j\rbrace & \text{ if }\lbrace i, j\rbrace\in\mathcal F\\
C\cup\lbrace i,j\rbrace & \text{ if }\lbrace i,j\rbrace\notin\mathcal F,	
\end{cases}
$	
\end{center}
where $C$ is the minimal element of the chevron assigned to the pair not in the family. By what we just proved above, this is a well-defined function from this set of pairs to $\mathcal F$.\\
Because in the second case $C\parallel\lbrace i,j\rbrace$, $|C\cup\lbrace i, j\rbrace|\geq 3$, and thus, in order to prove injectivity, we need to show that $\bar{C}\cup\lbrace i_1, j_1\rbrace\neq C\cup\lbrace i, j\rbrace$ for any two pairs $\lbrace i, j\rbrace\neq\lbrace i_i, j_1\rbrace$ with the desired property that are not in our family $\mathcal F$.\\
Assume that we do have $\bar{C}\cup\lbrace i_1, j_1\rbrace=C\cup\lbrace i, j\rbrace$ for two different pairs.
\begin{enumerate}
\item Case 1: $|C|=|\bar{C}|$. If $C=\bar C$ and $i$ is the element not in $C$, then $i$ is an element of the other pair. Now we have the equality $C\cup\lbrace i, j\rbrace=C\cup\lbrace i, j_1\rbrace$ with $j\neq j_1$. This immediately implies that $j, j_1\in C$, and if $(A,B,C)$ is the chevron assigned to $\lbrace i, j\rbrace$, then the same chevron can be assigned to $\lbrace i, j_1\rbrace$, contradicting the uniqueness of the chevrons for this type of pairs. This is because $\lbrace i, j_1\rbrace\subseteq C\cup\lbrace i, j\rbrace\subset A,B$ and $C\parallel\lbrace i,j_1\rbrace$ as $i\notin C$. Therefore $C\neq\bar C$ and since they have the same cardinality, $C\parallel\bar C$. But if $(A,B,C)$ is the chevron corresponding to $\lbrace i, j\rbrace$, then $\bar C\subset C\cup\lbrace i,j\rbrace\subseteq A,B$, thus $A,B,C,\bar C$ forms a butterfly in $\mathcal F$, which is a contradiction.  
\item Case 2: $|C|\neq|\bar C|$ and wlog, $|C|<|\bar C|$. Because adding the pair increases the size by at least one and at most two, we find that $|\bar C|=|C|+1$. This also means that $C\cap\lbrace i, j\rbrace=\emptyset$ and $|\bar C\cap\lbrace i_1, j_1\rbrace|=1$. If $C\parallel\bar C$, then we form a butterfly in $\mathcal F$ with the chevron where $\bar C$ is coming from. If they are comparable, then $C\subset\bar C$ and consequently $\lbrace i,j\rbrace\parallel \bar C$. Moreover, if $(\bar A, \bar B, \bar C)$ is the chevron for $\lbrace i_1, j_1\rbrace$, then $\lbrace i, j\rbrace\subset\bar C\cup\lbrace i_i, j_1\rbrace\subseteq\bar A, \bar B$, and thus $\lbrace i, j\rbrace$ can be assigned the chevron $(\bar A, \bar B,\bar C)$ and the size of $\bar C$ is strictly greater than the size of $C$, which contradict the choice of the chevron. 
\end{enumerate}
Therefore our function is an injection from the set of pairs containing at least one singleton from the family, to elements of $\mathcal F$. We have exactly $\binom{k}{2}+k(n-k)$ such pairs, giving $|\mathcal F|\geq\binom{k}{2}+k(n-k)$ as claimed.
\end{proof}
\end{theorem}
\par Note that if the number of singletons in $\mathcal F$ is $\Theta(n^{\alpha})$ for some $\alpha\in(0,1)$, Theorem 3 gives us a better bound than both Lemma 1 and Theorem 2. Lemma 1 gives us $\Theta(n^{2\alpha})$ elements in $\mathcal F$ and Theorem 2 gives $n+1$. On the other hand, Theorem 3 gives us $\Theta(n^{1+\alpha})$ elements in $\mathcal F$, which beats both of the previous bounds. 

\section{An improved bound on $\text{sat}^*(n,\mathcal N)$}
We will show that the saturation number for the poset $\mathcal N$ is at least $\sqrt n$. The key point is that in every $\mathcal N$-saturated family we can find an ordered pair $(F,G)$ such that $F\setminus G=\lbrace i\rbrace$ for every $i\in[n]$. This approach was also used by Martin, Smith and Walker \cite{martin2019improved} in their analysis of the saturation number of another poset, the diamond.
\begin{proposition} Let $\mathcal{F}$ be a $\mathcal{N}$-saturated family. Then $|\mathcal{F}|\geq\sqrt{n}$.
\begin{proof} We will show that for any $F$ in the family and for every $i\in F$, there exist sets $A$ and $B$ in $\mathcal{F}$ such that $A\subseteq F$ and $A\setminus B=\lbrace i\rbrace$. Since the poset $\mathcal N$ is invariant under taking complements, this will also tell us that for every $j\notin F$, there exists two sets $C$ and $D$ in $\mathcal F$ such that $F\subseteq C$ and $D\setminus C=\lbrace j\rbrace$. From this it would immediately follow that $|\mathcal F|\geq\sqrt n$, since by fixing an $F\in\mathcal F$ we can assign an ordered pair $(A,B)$ to every $i\in\lbrack n\rbrack$ with the property that $A\setminus B=\lbrace i\rbrace$, and $A,B\in\mathcal F$.\\
\par Let $F\in\mathcal F$ and $i\in F$. If there exist a set $A\in\mathcal F$ with $A\subseteq F$, $i\in A$ and $A\setminus\lbrace i\rbrace\in\mathcal F$, then we are done. Now suppose that no such $A$ exists and consider an element of the set $\lbrace A\in\mathcal F:i\in A, A\subseteq F\rbrace$ of minimal size, which we call $F^*$. We have that $F^*\setminus\lbrace i\rbrace\notin\mathcal F$ and thus it has to form a copy of $\mathcal N$ with thee other elements of $\mathcal F$.\begin{enumerate}
\item Case 1: $F^*\setminus\lbrace i\rbrace$ is one of the maximal elements, then we are in one of the two cases shown below.\\
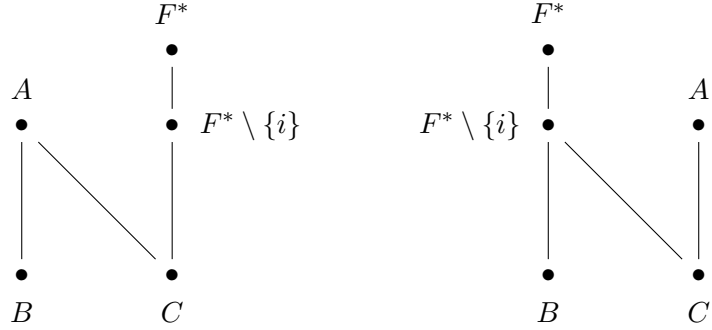
\begin{center}
\begin{tikzpicture}
\node (extra) at (1,2) {\textbullet};
\node [above=0cm of extra] {$F^*$};

\node (max1) at (-1,1) {\textbullet};

\node [above=0cm of max1] {$A$};

\node (max21) at (6,1) {\textbullet};

\node [left=0cm of max21] {$F^{*}\setminus\lbrace i\rbrace$};

\node (max22) at (8,1) {\textbullet};
\node [above=0cm of max22] {$A$};
\node (min21) at (6,-1) {\textbullet};
\node [below=0cm of min21] {$B$};
\node (min22) at (8,-1) {\textbullet};
\node [below=0cm of min22] {$C$};
\node (max2) at (1,1) {\textbullet};
\node [right=0cm of max2]  {$F^{*}\setminus\lbrace i\rbrace$};
\node (min1) at (-1,-1) {\textbullet};
\node [below=0cm of min1] {$B$};
\node (min2) at (1,-1) {\textbullet};
\node [below=0cm of min2] {$C$};
\node (extra2) at (6,2) {\textbullet};
\node [above=0cm of extra2] {$F^*$};

\draw (min1)--(max1)--(min2)--(max2)--(extra);
\draw (min21)--(max21)--(min22)--(max22);
\draw (extra2)--(max21);
\end{tikzpicture}\end{center}
Because we cannot have $A,B,C,F^{*}$ forming a copy of $\mathcal N$ in $\mathcal F$, it follows that in both cases $A$ and $F^{*}$ are comparable. We cannot have $F^{*}\subseteq A$ as $F^{*}\setminus\lbrace i\rbrace\parallel A$, therefore we must have $A\subseteq F^{*}$. Because $F^{*}\setminus\lbrace i\rbrace\parallel A$, $i\in A$ and $A\neq F^{*}$. This implies that $A$ is a proper subset of $F^*$ and thus of $F$, and element of the family containing $i$. This contradicts minimality of $F^{*}$.
\item Case 2: $F^{*}\setminus\lbrace i\rbrace$ is one of the minimal elements, then we are in the following two cases shown below.\\

\begin{center}\begin{tikzpicture}
\node (max1) at (-1,1) {\textbullet};
\node [above=0cm of max1] {$A$};
\node (max21) at (6,1) {\textbullet};
\node [above=0cm of max21] {$C$};
\node (max22) at (8,1) {\textbullet};
\node [above=0cm of max22] {$A$};
\node (min21) at (6,-1) {\textbullet};
\node [below=0cm of min21] {$B$};
\node (min22) at (8,-1) {\textbullet};
\node [below=0cm of min22] {$F^{*}\setminus\lbrace i\rbrace$};
\node (f1) at (7,-2.5) {Figure 2};
\node (f2) at (0,-2.5) {Figure 1};
\node (max2) at (1,1) {\textbullet};
\node [above=0cm of max2] {$B$};
\node (min1) at (-1,-1) {\textbullet};
\node [below=0cm of min1] {$F^{*}\setminus\lbrace i\rbrace$};
\node (min2) at (1,-1) {\textbullet};
\node [below=0cm of min2] {$C$};

\draw (min1)--(max1)--(min2)--(max2);
\draw (min21)--(max21)--(min22)--(max22);
\end{tikzpicture}\end{center}
Similarly as before, $A, F^{*}, B, C$ does not form a copy of $\mathcal N$.\\
\\
In the first case (Figure 1), if $i\notin A$, then $F^*\setminus A=\lbrace i\rbrace$, which gives the pair $(F^*,A)$.\\
If $i\in A$, then $F^*\subseteq A$, and either $A=F^*$ or $F^*\subset A$.\\If $A=F^{*}$, then $C$ and $F^*\setminus\lbrace i\rbrace$ are incomparable, while $C\subset F^*$. Thus $i\in C$, contradicting the minimality of $F^*$.\\If $F^*\subset A$ then, in order not to create a copy of $\mathcal N$ in $\mathcal F$, $F^*$ must be comparable to at least one of $B$ or $C$ (which are different from $F^*$ since they are incomparable to $F^*\setminus\lbrace i\rbrace$). Thus $B\subset F^*$ or $C\subset F^*$. Since $C\subset B$, we can assume wlog that $C\subset F^*$, which implies that $i\in C$, contradicting the minimality of $F^*$ again.\\
\\
In the second case (Figure 2), if $i\notin A$ or $i\notin C$, then similarly as above, we would find the pair $(F^*, A)$ or $(F^*, C)$. So we can assume that $i\in A\cap C$.\\If $F^*$ is different from both $A$ and $C$, then to avoid a copy of $\mathcal N$, $F^*$ and $B$ have to be comparable, and since $B\parallel F^*\setminus\lbrace i\rbrace$, we have to have $B\subset F^*$. This gives again $i\in B$ and thus contradicting minimality of $F^*$.\\
If $C=F^*$, then we find that $B\subset F^*$ and $i\in B$, leading to the same contradiction. If $A=F^*$, then $F^*\parallel C$ and $F^{*}\setminus\lbrace i\rbrace\subset C$, so $i\notin C$ and $F^*\setminus C=\lbrace i\rbrace$, which gives the pair $(F^*, C)$ and completes the proof.
\end{enumerate}
\end{proof}
\end{proposition}
\section{Extensions and further work}
In this final section we look at two more general families of posets that include the butterfly, and construct general saturated families.\\
\par We call the poset having two maximal incomparable elements and $k$ minimal incomparable elements, all of which are less than both maximal elements, a $K_{2,k}$.
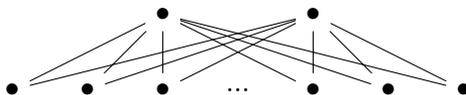
\begin{center}
\begin{tikzpicture}
\node (m1) at (-1,0) {$\bullet$};
\node (m2)at (1,0) {$\bullet$};
\node (m3) at (-3,-1) {$\bullet$};
\node (m4) at (-2,-1) {$\bullet$};
\node (m6) at (-1,-1) {$\bullet$};
\node (m5) at (0,-1) {...};
\node (m7) at (1,-1) {$\bullet$};
\node (m8) at (2,-1) {$\bullet$};
\node (m9) at (3,-1) {$\bullet$};
\draw (m1)--(m3)--(m2)--(m4)--(m1)--(m6)--(m2)--(m7)--(m1)--(m8)--(m2)--(m9)--(m1);
\end{tikzpicture}
\end{center}
We are now interested in the induced saturated number of a $K_{2,k}$ with ground set $\lbrack n\rbrack$. First, we construct a $K_{2,k}$ saturated family of size $O(n^k)$ with a special structure that will be used later.
\begin{proposition}
$\text{sat}^*(n,K_{2,k})=O(n^k)$.	
\end{proposition}
\begin{proof} We start with $\mathcal F_0$ consisting of all singletons and the chain $\emptyset\cup\lbrace\lbrace 1,2,...,t\rbrace:1\leq t\leq n\rbrace$. It is easy to check that it is $K_{2,k}$-free.\\We first prove that If $M$ is a set of size greater than $k$ and $M\notin\mathcal F_0$, then $\mathcal F_0\cup M$ contains a $K_{2,k}$.\\
Let $t$ be the smallest element not in $M$ and $t_1, t_2,...t_k$ be elements of $M$, none of which is the maximum of $M$. Furthermore, assume that $t_k$ is the maximum of the $k$ elements listed above.\\
If $t=1$, then $M$, $\lbrace 1,2,...t_k\rbrace$ and all singletons $\lbrace t_i\rbrace$ $1\leq i\leq k$ form a $K_{2,k}$. This is obvious since the singletons form an antichain of size $k$, they are all contained in both $M$ and $\lbrace 1,2,...t_k\rbrace$, and $M$ and $\lbrace 1,2,...,t_k\rbrace$ are incomparable since $M$ does not contain 1 and $\lbrace 1,2,...,t_k\rbrace$ does not contain the maximum of $M$.\\
If $t\neq 1$, then $t<$ max($M$) as $M$ is not in $\mathcal F_0$. Let $m=\text{max}\lbrace t,t_k\rbrace<$ max($M$).\\
We then have the following $K_{2,k}$: $M$, $\lbrace 1,2,...m\rbrace$ and $\lbrace t_i\rbrace$, $1\leq i\leq k$. By a similar argument as above, $M$ and $\lbrace 1,2,...m\rbrace$ contain all singletons and are incomparable since $t\notin M$ and $\text{max}(M)\notin\lbrace 1,2,...m\rbrace$.\\
Now let us list the sets of size at most $k$ that are not currently in $F_0$: $M_1$, $M_2$,..., $M_N$.\\We build a $K_{2,k}$-saturated family as follows. We start with $\mathcal F_0$. If $\mathcal F_0\cup M_1$ contains a $K_{2,k}$, then we do not add the set to our family, but if it does not, then we add it. We continue this procedure until we reach the end of the list. Let $\mathcal F$ be the family obtained in the end. It is $K_{2,k}$ free by construction and also saturated as we showed that if $|M|>k$, then $M\cup\mathcal F_0$ contains a $K_{2,k}$, and also the last step ensured that if we add $|M|<k$ to $\mathcal F$, then we form a $K_{2,k}$. Observe that all elements of $\mathcal F$, apart from the original chain appearing in $\mathcal F_0$, have cardinality at most $k$, thus\begin{center}
$\text{sat}^*(n,K_{2,k})\leq|\mathcal F|\leq\displaystyle\sum_{i=0}^k\binom{n}{i}+n-k=O(n^k).$\end{center}
\end{proof}
\par We can take a step forward and look at the following symmetric poset, which seems to be an even more natural generalisation of the butterfly.
\par We call the poset having $k$ incomparable maximal elements and $k$ incomparable minimal elements, all of which are less than all $k$ maximal elements, a $K_{k,k}$.

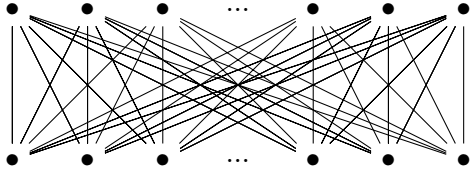
\begin{center}
\begin{tikzpicture}
\node (m1) at (-1,0) {$\bullet$};
\node (m2)at (1,0) {$\bullet$};
\node (m3) at (-3,-2) {$\bullet$};
\node (m4) at (-2,-2) {$\bullet$};
\node (m6) at (-1,-2) {$\bullet$};
\node (m5) at (0,-2) {...};
\node (m7) at (1,-2) {$\bullet$};
\node (m8) at (2,-2) {$\bullet$};
\node (m9) at (3,-2) {$\bullet$};
\node (m10) at (0,0) {...};
\node (m11) at (-3,0) {$\bullet$};
\node (m12) at (-2,0) {$\bullet$};
\node (m13) at (2,0) {$\bullet$};
\node (m14) at (3,0) {$\bullet$};
\draw (m1)--(m3)--(m2)--(m4)--(m1)--(m6)--(m2)--(m7)--(m1)--(m8)--(m2)--(m9)--(m1);
\draw (m11)--(m3)--(m11)--(m4)--(m11)--(m6)--(m11)--(m7)--(m11)--(m8)--(m11)--(m9);
\draw (m12)--(m3)--(m12)--(m4)--(m12)--(m6)--(m12)--(m7)--(m12)--(m8)--(m12)--(m9);
\draw (m13)--(m3)--(m13)--(m4)--(m13)--(m6)--(m13)--(m7)--(m13)--(m8)--(m13)--(m9);
\draw (m14)--(m3)--(m14)--(m4)--(m14)--(m6)--(m14)--(m7)--(m14)--(m8)--(m14)--(m9);
\end{tikzpicture}
\end{center}
\begin{proposition}$\text{sat}^*(n, K_{k,k})=O(n^{2k-2})$.\begin{proof} We start with $\mathcal F_0$ consisting of all the singletons and the following $k-1$ chains
\begin{center}
$\mathcal C_1:$ $2,3,4,\dotsc,n,1$\\$\mathcal C_2:$
$1,3,4,\dotsc,n,2$\\$\vdots$\\$\mathcal C_{k-1}: $$1,2,3,\dotsc,n,k-1$,
\end{center}
where by the chain $a_1,a_2,\dotsc,a_n$ we mean the chain $\emptyset, \lbrace a_1\rbrace, \lbrace a_1, a_2\rbrace,\dotsc,\lbrace a_1, a_2,\dotsc,a_n\rbrace$. It is clear that $\mathcal F_0$ is a $K_{k,k}$-free family since the maximal elements cannot be singletons, but then by the pigeonhole principle at least two will have to be in the same chain, which is impossible since they have to form an antichain.\\
We have seen in the construction of a $K_{2,k}$-saturated family that if we have a chain and a set $M$ of size greater than $k$, we can construct a $K_{2,k}$ with $M$, $k$ arbitrary singletons of $M$ and one element of the chain, as long as those singletons of do not contain the maximum element of $M$ with respect to the order induced by the chain.\\
Assume now that $M$ is a set of size at least $2k-1$ not in $\mathcal F_0$. Let $S$ be the set of maximal elements of $M$, with respect to the $k-1$ orders induced by the above chains. We have $|M\setminus S|\geq k$, thus we can select $t_1, t_2,\dotsc, t_k$ elements of $M$, none of which is the maximum with respect to any of the $k-1$ orders. Therefore, our previous construction gives us $k-1$ sets $A_1, A_2,\dotsc, A_{k-1}$ with the property that $\lbrace t_1, t_2,\dotsc,t_k\rbrace\subset A_i$, $A_i\parallel M$ and $A_i\in\mathcal C_i$ for all $1\leq i\leq k-1$. \\
Observe also that since $t_i$ is not among the maximums, $t_i\geq k$ for every $i$, and consequently, $t_1, t_2,\dotsc t_k$ will have the same order in all of the above chains, which is just the usual order. Let $t_k$ be the biggest of them.\\
If $M$ contains $1,2,\dotsc, k-1$, then it contains all maximal elements of the $k-1$ chains, which by our previous construction means that $i\notin A_i$ and we can replace $A_i$ with $\lbrack n\rbrack-\lbrace i\rbrace\in\mathcal C_i$. We then have $M$, $\lbrack n\rbrack-\lbrace i\rbrace, \lbrace t_i\rbrace$, $1\leq i\leq k-1$ forming a copy of $K_{k,k}$.\\
If $M$ does not contain 1, then the sets that our previous construction gives us for the chains $\mathcal C_i$, $i\geq2$ are $\lbrace 1,3,4,\dotsc t_k\rbrace$, $\lbrace 1,2,4,\dotsc t_k\rbrace,\dotsc,\lbrace 1,2,3,\dotsc t_k\rbrace$. For $\mathcal C_1$ our construction gives $\lbrace 2,3,\dotsc \text{max}(t_k,t)\rbrace$, where $t$ is the smallest one not in $M$, with respect to the first order. The first $k-2$ sets have the same size, $t_k-1$, but are different, thus incomparable. The set obtained from the first chain has size at least $t_k-1$, but it does not contain 1, so it cannot be comparable to any of the above. Therefore $M$ together with these $k-1$ sets and the $k$ singletons forms a $K_{k,k}$.\\
If $M$ contains 1, but it does not contain 2, then $t_k\geq 2$ and our construction gives $\lbrace 2,3,\dotsc t_k\rbrace$, $\lbrace 1,2,4,\dotsc t_k\rbrace,\dotsc,\lbrace 1,2,3,\dotsc t_k\rbrace$ for chains $\mathcal C_i$, $i\neq 2$.\\For $\mathcal C_2$ we get $\lbrace 1,3,4,\dotsc,\text{max}(t,t_k)\rbrace$. Once again, the first $k-2$ sets are different and have the same size, while the last set has size greater or equal to the size of the others, but it does not contain 2, while all the other set all contain it. Hence, together with $M$, they form an antichain of size $k$, and together with the $k$ singletons form a copy of $K_{k,k}$. \\
It can be shown inductively (if $M$ contains $1,2,\dotsc,t$, but not $t+1$) that we always form a copy of $K_{k,k}$.\\
We therefore have a $K_{k,k}$-free family which forms a $K_{k,k}$ with any set $M\notin \mathcal F_0$ with cardinality at least $2k-1$. Thus, the same argument as for for the $K_{2,k}$, gives that\begin{center}
$\text{sat}^*(n,K_{k,k})\leq\displaystyle\sum_{i=0}^{2k-2}\binom{n}{i}+(k-1)(n-2k+1)=O(n^{2k-2})$.\end{center}
\end{proof}
\end{proposition}
\par We can see that the singleton analysis is not as helpful as it was for the butterfly since increasing the number of incomparable elements increases the case analysis too. Moreover, if the aim is to prove a lower bound of order at least $n^k$ using these methods, we will need to use sets of size $k$ rather than singletons or pairs and we immediately loose control. The power of the singletons stays in the property that if they are incomparable to a non-empty set, then they are disjoint from that set. There is no obvious strong enough analogue for sets of size $k$.
\par However, as illustrated in Theorem 3, the importance of pairs for the butterfly is perhaps evidence that the conjecture of Ferrara, Kay, Kramer, Martin, Reiniger, Smith and Sullivan is true:
\begin{conjecture}[\cite{Ferrara2017TheSN}] $sat^*(n,\mathcal B)=\Theta(n^2)$.	
\end{conjecture}
We have also seen in this final section that starting with just the singletons and a saturated chain we could build a $K_{2,k}$-saturated family by adding sets of size at most $k$. The construction was powerful enough to allow us to consider only the lower levels of the power set, thus leading us to conjecture the following:
\begin{conjecture}
$sat^*(n,K_{2,k})=\Theta(n^k)$.	
\end{conjecture}
Our construction for the $K_{k,k}$-saturated family relied heavily on our previous observations about the $K_{2,k}$. By choosing chains that induced incompatible enough orderings of the ground set, we were able to adapt the construction of a $K_{2,k}$ saturated family and construct a $K_{k,k}$-saturated family, concentrated again at lower levels of the power set. Consequently, we conjecture that the following statement is also true:
\begin{conjecture}
$sat^*(n,K_{k,k})=\Theta(n^{2k-2})$.	
\end{conjecture}
\bibliography{butterfly}
\bibliographystyle{ieeetr}
\end{document}